\documentclass[10pt]{amsart}
\usepackage{amsthm, amsfonts, amssymb, amsmath, graphicx, enumitem, bbold, caption}
\usepackage[margin=4cm]{geometry}
\usepackage{lmodern}

\newcommand{\C}{\mathbb{C}}

\newcommand{\Z}{\mathbb{Z}}

\newcommand{\F}{\mathbb{F}}

\newcommand{\SL}{\mathrm{SL}}

\newcommand{\noin}{\noindent}

\theoremstyle{definition}
\newtheorem{thm}{Theorem}
\newtheorem{lem}[thm]{Lemma}

\begin{document}
\title{On bounded elementary generation for $\SL_n$\\ over polynomial rings}
\author{Bogdan Nica}
\date{January 29, 2017}
\address{\newline Department of Mathematics and Statistics \newline McGill University, Montreal}

\begin{abstract}
Let $\F[X]$ be the polynomial ring over a finite field $\F$. It is shown that, for $n\geq 3$, the special linear group $\SL_n(\F[X])$ is boundedly generated by the elementary matrices.
\end{abstract}

\maketitle
\section{Introduction}

The special linear group $\SL_n(\Z)$ is generated by the elementary matrices, that is, matrices which differ from the identity by at most one non-zero off-diagonal entry. Far more remarkable is the following fact: 

\begin{thm}[Carter - Keller \cite{CK}]\label{thm: ck}
Let $n\geq 3$. Then $\SL_n(\Z)$ is boundedly generated by the elementary matrices.
\end{thm}

This means that, for some positive integer $\nu_n$, every matrix in $\SL_n(\Z)$ is a product of at most $\nu_n$ elementary matrices. The Carter - Keller theorem provides, in fact, the explicit bound $\nu_n=\tfrac{1}{2}(3n^2-n)+36$. See \cite{AM} for a variation on the Carter - Keller argument, with a slightly worse bound. In \cite{CK2}, Carter and Keller extend their argument to rings of integers in algebraic number fields.

A different approach to bounded elementary generation for $\SL_n$ over rings of integers, based on unpublished work of Carter, Keller, and Paige, can be found in \cite{WM}. The novelty is the use of model-theoretic ideas. Unlike the original Carter - Keller approach, this is a non-explicit argument. One proves the \emph{existence} of a bound on the number of elementary matrices needed to express a matrix in $\SL_n$.

Elementary generation of $\SL_n$ also holds for polynomial rings over fields. However, bounded elementary generation may fail. In \cite{K} van der Kallen shows, by means of algebraic K-theory, that $\SL_n(\C[X])$, $n\geq 2$, is not boundedly generated by the elementary matrices. From an arithmetical viewpoint, however, the closest relative of $\Z$ is a polynomial ring over a \emph{finite} field. The purpose of this note is to show the following analogue of Theorem~\ref{thm: ck}, which appears to be new (cf., e.g., \cite[p.523]{S}).

\begin{thm}\label{thm: F}
Let $\F$ be a finite field, and let $n\geq 3$. Then $\SL_n(\F[X])$ is boundedly generated by the elementary matrices.
\end{thm}

The proof is an adaptation of the Carter - Keller argument. Here is one technical difference. A crucial role in \cite{CK}, and also in \cite{AM}, is played by a `power lemma' \cite[Lemma 1]{CK} whose origins lie in properties of the so-called Mennicke symbols. We use instead a simple `swindle lemma', Lemma~\ref{lem: s} herein. A version of this swindle was used in \cite[\S 2.3]{N}. The proof  of Theorem~\ref{thm: F} yields the explicit bound $\nu_n=\tfrac{1}{2}(3n^2-n)+29$. As this bound does not depend on the size of the finite field $\F$, it follows that Theorem~\ref{thm: F} holds, more generally, whenever $\F$ is an algebraic extension of a finite field.

One cannot take $n=2$ in Theorem~\ref{thm: F}: $\SL_2(\F[X])$ is not boundedly generated by the elementary matrices. This fact, and the reason behind it, are analogous to what happens for $\SL_2(\Z)$. The principal congruence subgroup of $\SL_2(\F[X])$ corresponding to the ideal $(X)$, in other words the kernel of the surjective homomorphism $\SL_2(\F[X])\to \SL_2(\F)$ given by the evaluation $X=0$, has a free product structure.

\section{Proof of Theorem~\ref{thm: F}}\label{sec: Z}
Throughout, an elementary operation will be called, simply, a move. We allow the degenerate move of multiplying by the identity matrix. We write $\sim$ for the equivalence relation of being connected by a finite number of moves. 

\subsection{Reduction to a framed $\SL_2$ matrix} \label{sec: one} The first step is to reduce a matrix in $\SL_n$, $n\geq 3$, to a matrix of the following form: 
\begin{align*}
\begin{pmatrix} a & b &\\ c & d &\\
& & I_{n-2}
\end{pmatrix}
\end{align*}
This is a standard reduction which works over any principal domain $A$. The general concept underpinning this procedure is Bass's stable range \cite{Ba}. For the sake of completeness, let us sketch the argument for $n=3$. Let $(u,v,w)$ be the last row of a matrix in $\SL_3(A)$. Thus, $u$, $v$, and $w$ are relatively prime, and we may assume that either $u$ or $v$ is non-zero. A suitable move takes us to a matrix whose last row is $(u',v',w)$, and such that $u'$ and $v'$ are relatively prime. The key fact here is that, if $\gcd (u,v,w)=1$ and $u$ is non-zero, then $\gcd(u,v+tw)=1$ for some $t\in A$. (An explicit choice for $t$ is the product of all primes dividing $u$ but not $v$. More precisely, we take one prime per associate class. We set $t=1$ if there are no such primes.) Now $w-1$ is a combination of $u'$ and $v'$, so two moves turn $w$ into $1$. Four additional moves clear the last row and the last column. In summary, we have reached a framed $\SL_2$ matrix, as desired, in $7$ moves. More generally, this argument reduces an $\SL_n$ matrix to a framed $\SL_2$ matrix in $\tfrac{1}{2}(3n^2-n)-5$ moves. 

The remainder of the argument is devoted to showing that $34$ moves are sufficient in order to reduce, in $\SL_3$, a framed $\SL_2$ matrix to the identity. For the purposes of the next step, we assume that $a\neq 0$; otherwise, the reduction is trivial and quick, in only $3$ moves.

\subsection{A convenient anti-diagonal} The second step will use the following analogue of Dirichlet's theorem on primes in arithmetic progressions. 

\begin{thm}[Kornblum - Artin]\label{thm: KA}
If $a,b\in \F[X]$ are relatively prime and $a\neq 0$, then there are infinitely many primes congruent to $b$ mod $a$. Furthermore, such a prime can have arbitrary degree, provided the degree is sufficiently large.
\end{thm}

The first part is due to Kornblum (1919). The second part is a sharpening due to Artin (1921). See \cite[Chapter 4]{R} for a modern treatment.

Consider a matrix
\begin{align*}
\begin{pmatrix} a & b\\ c & d
\end{pmatrix}\in \SL_2(\F[X]).
\end{align*}
As $a$ and $b$ are relatively prime, the first part of Theorem~\ref{thm: KA} ensures that there is a prime $b'\in \F[X]$ satisfying $b'\equiv b$ mod $a$. Similarly, there is a prime $c'\in \F[X]$ satisfying $c'\equiv c$ mod $a$. Thus
\begin{align*}
\begin{pmatrix} a & b\\ c & d 
\end{pmatrix} \;\sim\;
\begin{pmatrix} a & b'\\ c' & d' 
\end{pmatrix}
\end{align*}
in $2$ moves. Furthermore, we can assume that $b'$ and $c'$ have relatively prime degrees: once $b'$ has been chosen, we use the second part of Theorem~\ref{thm: KA} to pick $c'$ of suitable degree.

\subsection{The main step} Let
\begin{align*}
\begin{pmatrix} a & b\\ c & d
\end{pmatrix}\in \SL_2(\F[X])
\end{align*}
be a matrix enjoying the property granted by the previous step: the anti-diagonal entries $b$ and $c$ are prime, with relatively prime degrees.  

Let $q$ denote the number of elements in $\F$. Then the integers
\begin{align*}
\delta(b):=\frac{q^{\deg(b)}-1}{q-1}, \qquad \delta(c):=\frac{q^{\deg(c)}-1}{q-1}
\end{align*}
are relatively prime, as well. Let $x$ and $y$ be positive integers satisfying $x\delta(b)-y\delta(c)=1$. We write
\begin{align*}
\begin{pmatrix} a & b\\c & d 
\end{pmatrix}=XY^{-1},
\end{align*} 
where
\begin{align*}
X=\begin{pmatrix} a & b\\ c & d 
\end{pmatrix}^{x\delta(b)},\quad Y=\begin{pmatrix} a & b\\ c & d 
\end{pmatrix}^{y\delta(c)}.
\end{align*}
We aim to reduce $X$ and $Y$ independently in $\SL_3$. More precisely, we will show that
\begin{align*}\tag{$\dagger$}
\begin{pmatrix} Y & \\  & 1
\end{pmatrix}
\;\sim\;
\begin{pmatrix} D(u) & \\  & -1
\end{pmatrix}, \qquad 
D(u):=\begin{pmatrix} -u & \\ & u^{-1}
\end{pmatrix}
\end{align*}
in $14$ moves, for some $u\in \F^*$. The same will hold for $Y^{-1}$ in place of $Y$, and $u^{-1}$ in place of $u$, by inverting. It also holds for $X$ in place of $Y$, by interchanging $b$ and $c$, and then transposing, with respect to some other unit $v\in \F^*$. However, $D(v)\sim D(u)$ in $\SL_2$, in $4$ additional moves. We can then deduce that
\begin{align*}
\begin{pmatrix} XY^{-1} & \\  & 1
\end{pmatrix}=\begin{pmatrix} X & \\  & 1
\end{pmatrix}\begin{pmatrix} Y^{-1} & \\  & 1
\end{pmatrix}
\;\sim\;
\begin{pmatrix} D(u) & \\  & -1
\end{pmatrix}\begin{pmatrix} D(u^{-1}) & \\  & -1
\end{pmatrix}=I_3
\end{align*}
in $14+4+14=32$ moves. Along the way, we are using the fact that diagonal matrices normalize the elementary matrices. Taking into account the second step, we conclude that $34$ moves are sufficient in order to reduce a framed $\SL_2$ matrix to the identity.

Let us turn to proving $(\dagger)$. By the Cayley - Hamilton theorem, there are $e,f\in \F[X]$ such that:
\begin{align*}
\begin{pmatrix} a & b\\ c & d 
\end{pmatrix}^{y\delta(c)}=eI_2+f \begin{pmatrix} a & b\\ c & d 
\end{pmatrix} = \begin{pmatrix} e+fa & fb\\ fc & e+fd 
\end{pmatrix}
\end{align*}
Modulo $c$, the above matrices become upper triangular. So $e+fa\equiv a^{y\delta(c)}$ mod $c$. On the other hand, $a^{\delta(c)}$ mod $c$ is in $\F^*$. This follows by viewing the finite field $\F[X]/(c)$ as an extension of $\F$ of degree $\deg(c)$. Thus $e+fa\equiv u\in \F^*$ mod $c$. A similar argument applies to the lower diagonal entry. Keeping in mind that the determinant is $1$, we find that $e+fd\equiv u^{-1}\in \F^*$ mod $c$. At this point, we would like to replace the lower entry, $fc$, by $c$ so as to be able to perform reductions. 

These considerations motivate the following lemma. Roughly speaking, it provides a way of swindling factors across the diagonal.

\begin{lem}\label{lem: s}
Let $A$ be a principal domain, and let
\begin{align*}
\begin{pmatrix} a & b\\ sc & d 
\end{pmatrix}\in \SL_2(A)
\end{align*} 
where $a\equiv d$ mod $s$. Then
\begin{align*}
\begin{pmatrix} a & b &\\ sc & d &\\ & & 1
\end{pmatrix}
\;\sim\;
\begin{pmatrix} \pm a & -sb &\\ c & \mp d &\\ & & -1
\end{pmatrix}
\end{align*}
in $11$ moves.
\end{lem}

\begin{proof}
The degenerate case $s=0$ is easily seen to hold, so let us assume that $s\neq 0$. The hypotheses imply that $a^2\equiv ad\equiv 1$ mod $s$. So there are $s_1, s_2$ and $k_1,k_2$ in $A$ such that
\begin{align*}
s=s_1s_2, \qquad a=k_1s_1+1=k_2s_2-1.
\end{align*}
Now
\begin{align*}
\begin{pmatrix} a & b & 0\\ sc & d & 0\\ 0 & 0 & 1
\end{pmatrix}
\;\sim\;\begin{pmatrix} a & b & 0\\ sc & d & 0\\ s_1 & 0 & 1
\end{pmatrix}
\;\sim\; \begin{pmatrix} 1 & b & -k_1\\ 0 & d & -s_2c\\ s_1 & 0 & 1
\end{pmatrix}
\;\sim\; \begin{pmatrix} 1 & b & -k_1\\ 0 & d & -s_2c\\ 0 & -s_1b & a
\end{pmatrix}
\end{align*}
by a column move, two row moves, and one more row move. We have basically swindled $s_1$ across the diagonal, and we now go for $s_2$. Firstly, 
\begin{align*}
\begin{pmatrix} 1 & b & -k_1\\ 0 & d & -s_2c\\ 0 & -s_1b & a
\end{pmatrix}
\;\sim\;
\begin{pmatrix} 1 & 0 & s_2\\ 0 & d & -s_2c\\ 0 & -s_1b & a
\end{pmatrix}
\end{align*}
by two column moves. Next,
\begin{align*}
\begin{pmatrix} 1 & 0 & s_2\\ 0 & d & -s_2c\\ 0 & -s_1b & a
\end{pmatrix}
&\;\sim\;
\begin{pmatrix} 1 & 0 & s_2\\ c & d & 0\\ -k_2 & -s_1b & -1
\end{pmatrix}\\
&\;\sim\;
\begin{pmatrix} -a & -sb & 0\\ c & d & 0\\ -k_2 & -s_1b & -1
\end{pmatrix}
\;\sim\;
\begin{pmatrix} -a & -sb & 0\\ c & d & 0\\ 0 & 0 & -1
\end{pmatrix}
\end{align*}
by two row moves, another row move, and two column moves. Overall, we have performed $11$ moves, as claimed.

For the other choice of signs on the diagonal, one could `pivot' around $d$ instead of $a$. Alternatively, start from the above choice of signs, invert both matrices, interchange $a$ and $d$, and switch the signs of $b$ and $c$. 
\end{proof}

Applying the above lemma, we obtain
\begin{align*}
\begin{pmatrix} e+fa & fb &\\ fc & e+fd &\\ & & 1
\end{pmatrix}
\;\sim\;
\begin{pmatrix} -(e+fa) & -f^2b &\\ c & e+fd &\\ & & -1
\end{pmatrix}\;\sim\;
\begin{pmatrix} -u & \dots &\\ c & u^{-1} &\\ & & -1
\end{pmatrix}
\end{align*}
in $11+2=13$ moves. Taking the determinant reveals that the missing entry of the last matrix is $0$. One additional move, bringing the total to $14$, clears out the entry $c$. This completes the argument for $(\dagger)$, and so for Theorem~\ref{thm: F} as well.

\section{Further remarks}

\subsection{}\label{sury} The notion of bounded generation is commonly used for the property that a group is a product of finitely many cyclic subgroups. For $\SL_n(\Z)$, $n\geq 3$, bounded cyclic generation follows from bounded elementary generation. This is no longer the case over $\F[X]$. In fact, bounded cyclic generation fails for $\SL_n(\F[X])$, $n\geq 3$. The idea that bounded cyclic generation is essentially a characteristic $0$ phenomenon, is crystallized by the following result from \cite{A+}: if a linear group in positive characteristic enjoys bounded cyclic generation, then the group is virtually abelian.

\subsection{} Lemma~\ref{lem: s} can also be used over $A=\Z$. In this case, it leads to a simplification of the original Carter - Keller argument for $\SL_n(\Z)$, and to the better bound $\nu_n=\tfrac{1}{2}(3n^2-n)+25$. The improved bound is irrelevant from an asymptotic perspective, but it becomes interesting in the case $n=3$. The question, which seems to us quite appealing, is how many elementary operations are needed to reduce a matrix in $\SL_3(\Z)$ to the identity? Carter and Keller have shown that $48$ operations are sufficient. We have reduced this number to $37$. We challenge the reader to reduce this bound even further.

\subsection{} There is an interesting issue of effectiveness in using the Kornblum - Artin theorem. The usual Dirichlet theorem, used in \cite{CK}, is made effective by a result of Linnik and its modern improvements (see, for instance, \cite{HB}). Theorem~\ref{thm: KA} is also effective, since the Riemann hypothesis in the function field context is already known. 

See \cite{BS} for further instances of Dirichlet-type theorems over polynomial rings.

\bigskip
\noin\textbf{Acknowledgements.} I would like to thank Dave Witte Morris for thoughtful comments, and for pointing out Theorem~\ref{thm: KA} and reference \cite{R}. I am also grateful to B. Sury for many valuable remarks--notably, \ref{sury} herein.

\end{document}